\documentclass{amsart}
\usepackage{setspace}
\usepackage{a4}
\usepackage{amsthm,latexsym}
\usepackage{amsfonts}
\usepackage{graphicx}
\usepackage{textcomp}
\usepackage{cite}
\usepackage{enumerate}
\usepackage[mathscr]{euscript}
\usepackage{mathtools}
\newtheorem{theorem}{Theorem}[section]

\newtheorem{conjecture}[theorem]{Conjecture}

\newtheorem{definition}[theorem]{Definition}

\newtheorem{lemma} [theorem]{Lemma}

\title{This is the title}
\usepackage{amssymb}
\usepackage{amsmath}
\usepackage{tikz}
\usepackage{hyperref}
\usepackage{mathtools}
\usetikzlibrary{cd}
\usepackage{fancyhdr}

\begin{document}
	\hrule\hrule\hrule\hrule\hrule
	\vspace{0.3cm}	
	\begin{center}
		{\bf{NONCOMMUTATIVE DONOHO-STARK-ELAD-BRUCKSTEIN-RICAUD-TORR\'{E}SANI  UNCERTAINTY PRINCIPLE}}\\
		\vspace{0.3cm}
		\hrule\hrule\hrule\hrule\hrule
		\vspace{0.3cm}
		\textbf{K. MAHESH KRISHNA}\\
	School of Mathematics and Natural Sciences\\
	Chanakya University Global Campus \\
	Haraluru Village, Near Kempe Gowda International Airport (BIAL)\\
	Devanahalli Taluk, 	Bengaluru  Rural District\\
	Karnataka  562 110 India\\
	Email: kmaheshak@gmail.com\\
		
		Date: \today
	\end{center}

\hrule
\vspace{0.5cm}
\textbf{Abstract}: Let $\{\tau_n\}_{n=1}^\infty$ and $\{\omega_m\}_{m=1}^\infty$ be  two modular Parseval  frames  for a Hilbert C*-module  $\mathcal{E}$. Then for every $x \in \mathcal{E}\setminus\{0\}$,  we show that 
\begin{align}\label{UE}
	\|\theta_\tau x \|_0 	\|\theta_\omega x \|_0	\geq \frac{1}{\sup_{n, m \in \mathbb{N}} \|\langle \tau_n, \omega_m\rangle \|^2}.
\end{align}
 We call Inequality (\ref{UE}) as \textbf{Noncommutative Donoho-Stark-Elad-Bruckstein-Ricaud-Torr\'{e}sani Uncertainty Principle}. Inequality (\ref{UE}) is the noncommutative analogue of breakthrough   Ricaud-Torr\'{e}sani uncertainty principle \textit{[IEEE Trans. Inform. Theory, 2013]}. In particular, Inequality (\ref{UE}) extends   Elad-Bruckstein uncertainty principle \textit{[IEEE Trans. Inform. Theory, 2002]} and Donoho-Stark uncertainty principle \textit{[SIAM J. Appl. Math., 1989]}.


\vspace{0.5cm}
\hrule 

\section{Introduction}
In 1989,  Donoho and Stark derived following uncertainty principle which is one of the greatest inequality of all time in both pure and applied Mathematics  \cite{DONOHOSTARK}. For $h \in \mathbb{C}^d$, let $\|h\|_0$ be the number of nonzero entries in $h$. Let $\hat{}: \mathbb{C}^d \to  \mathbb{C}^d$ be the Fourier transform defined by 
\begin{align*}
	\widehat{(a_j)_{j=0}^{d-1}}\coloneqq \frac{1}{\sqrt{d}}\left(\sum_{j=0}^{d-1}a_je^{\frac{-2 \pi i j k}{d}}\right)_{k=0}^{d-1}, \quad \forall (a_j)_{j=0}^{d-1} \in \mathbb{C}^d.
\end{align*}
\begin{theorem}  (\textbf{Donoho-Stark Uncertainty Principle}) \cite{TERRAS, DONOHOSTARK} \label{DS}
	For every $d\in \mathbb{N}$, 
	\begin{align}\label{DSE}
		\left(\frac{\|h\|_0+\|\widehat{h}\|_0}{2}\right)^2	\geq \|h\|_0\|\widehat{h}\|_0	\geq d, \quad \forall h \in \mathbb{C}^d\setminus \{0\}.
	\end{align}
\end{theorem}
By noting that Fourier transform is  unitary and unitary operators are in one to one correspondence with orthonormal bases, in 2002, Elad and Bruckstein generalized Inequality (\ref{DSE}) to arbitrary orthonormal bases \cite{ELADBRUCKSTEIN}. To state the result we need some notations. Given a collection $\{\tau_j\}_{j=1}^n$ in a finite dimensional Hilbert space $\mathcal{H}$ over $\mathbb{K}$ ($\mathbb{R}$ or $\mathbb{C}$), we define 
\begin{align*}
	\theta_\tau: \mathcal{H} \ni h \mapsto \theta_\tau h \coloneqq (\langle h, \tau_j\rangle)_{j=1}^n \in \mathbb{K} ^n.
\end{align*}
\begin{theorem} (\textbf{Elad-Bruckstein Uncertainty Principle}) \cite{ELAD, ELADBRUCKSTEIN} \label{EB}
	Let $\{\tau_j\}_{j=1}^n$,  $\{\omega_j\}_{j=1}^n$ be two orthonormal bases for a  finite dimensional Hilbert space $\mathcal{H}$. Then 
	\begin{align*}
		\left(\frac{\|\theta_\tau h\|_0+\|\theta_\omega h\|_0}{2}\right)^2	\geq \|\theta_\tau h\|_0\|\theta_\omega h\|_0\geq \frac{1}{\max_{1\leq j, k \leq n}|\langle\tau_j, \omega_k \rangle|^2}, \quad \forall h \in \mathcal{H}\setminus \{0\}.
	\end{align*}
\end{theorem}
 In 2013, Ricaud and Torr\'{e}sani  showed that orthonormal bases in Theorem \ref{EB} can be improved to Parseval frames \cite{RICAUDTORRESANI}. Recall that a collection $\{\tau_j\}_{j=1}^n$ in a finite dimensional Hilbert space $\mathcal{H}$ is said to be a Parseval frame for $\mathcal{H}$ \cite{BENEDETTPFICKUS} if 
 \begin{align*}
 	\|h\|^2=\sum_{j=1}^n |\langle h, \tau_j\rangle |^2, \quad \forall h \in \mathcal{H}.
 \end{align*}
\begin{theorem} (\textbf{Ricaud-Torr\'{e}sani Uncertainty Principle}) \cite{RICAUDTORRESANI} \label{RT}
	Let $\{\tau_j\}_{j=1}^n$,  $\{\omega_j\}_{j=1}^n$ be two Parseval frames   for a  finite dimensional Hilbert space $\mathcal{H}$. Then 
	\begin{align*}
		\left(\frac{\|\theta_\tau h\|_0+\|\theta_\omega h\|_0}{2}\right)^2	\geq \|\theta_\tau h\|_0\|\theta_\omega h\|_0\geq \frac{1}{\max_{1\leq j, k \leq n}|\langle\tau_j, \omega_k \rangle|^2}, \quad \forall h \in \mathcal{H}\setminus \{0\}.
	\end{align*}	
\end{theorem}

The main purpose of this paper is to generalize and derive a noncommutative version of Theorem \ref{RT}. For this we want generalization of Hilbert spaces known as Hilbert C*-modules. Hilbert C*-modules are first introduced by Kaplansky \cite{KAPLANSKY} for modules over commutative C*-algebras and later developed for modules over arbitrary C*-algebras by Paschke  \cite{PASCHKE} and Rieffel \cite{RIEFFEL}.
\begin{definition}\cite{KAPLANSKY, PASCHKE, RIEFFEL}
	Let $\mathcal{A}$ be a  unital C*-algebra. A left module 	 $\mathcal{E}$  over $\mathcal{A}$ is said to be a (left) Hilbert C*-module if there exists a  map $ \langle \cdot, \cdot \rangle: \mathcal{E}\times \mathcal{E} \to \mathcal{A}$ such that the following hold. 
	\begin{enumerate}[\upshape(i)]
		\item $\langle x, x \rangle  \geq 0$, $\forall x \in \mathcal{E}$. If $x \in  \mathcal{E}$ satisfies $\langle x, x \rangle  =0 $, then $x=0$.
		\item $\langle x+y, z \rangle  =\langle x, z \rangle+\langle y, z \rangle$, $\forall x,y,z \in \mathcal{E}$.
		\item  $\langle ax, y \rangle  =a\langle x, y \rangle$, $\forall x,y \in \mathcal{E}$. $\forall a \in \mathcal{A}$.
		\item $\langle x, y \rangle=\langle y,x \rangle^*$, $\forall x,y \in \mathcal{E}$.
		\item $\mathcal{E}$ is complete w.r.t. the norm $\|x\|\coloneqq \sqrt{\|\langle x, x \rangle\|}$,  $\forall x \in \mathcal{E}$.
	\end{enumerate}
\end{definition}
We are going to use the following inequality.
\begin{lemma}\cite{PASCHKE}\label{CAUCHYSCHWARZ} (Noncommutative Cauchy-Schwarz inequality) If $\mathcal{E}$ is a Hilbert C*-module  over $\mathcal{A}$, then 
	\begin{align*}
		\langle x, y \rangle \langle y,x \rangle\leq \|\langle y, y \rangle\|\langle x, x \rangle, \quad \forall x,y \in \mathcal{E}.
	\end{align*}
\end{lemma}
Given   a unital   C*-algebra  $\mathcal{A}$, define 
\begin{align*}
	\ell^2(\mathbb{N}, \mathcal{A})\coloneqq \left\{ \{a_n\}_{n=1}^\infty: a_n \in \mathcal{A}, \forall n \in \mathbb{N}, \sum_{n=1}^{\infty}a_na_n^* \text{ converges in } \mathcal{A}\right\}.
\end{align*} 
 Modular $\mathcal{A}$-inner product on $\ell^2(\mathbb{N}, \mathcal{A})$ is defined as 
\begin{align*}
	\langle \{a_n\}_{n=1}^\infty, \{b_n\}_{n=1}^\infty\rangle \coloneqq \sum_{n=1}^{\infty}a_nb_n^*,\quad \forall  \{a_n\}_{n=1}^\infty, \{b_n\}_{n=1}^\infty \in 	\ell^2(\mathbb{N}, \mathcal{A}).
\end{align*}
Hence the norm on $	\ell^2(\mathbb{N}, \mathcal{A})$ becomes 
\begin{align*}
	\| \{a_n\}_{n=1}^\infty\|\coloneqq \left\|\sum_{n=1}^{\infty}a_na_n^*\right\|^\frac{1}{2}, \quad \forall  \{a_n\}_{n=1}^\infty \in 	\ell^2(\mathbb{N}, \mathcal{A}).
\end{align*}

\section{Noncommutative Donoho-Stark-Elad-Bruckstein-Ricaud-Torr\'{e}sani Uncertainty Principle}
	We start by recalling  the definition of Parseval frames for Hilbert C*-modules by Frank and Larson  \cite{FRANKLARSON}.
	\begin{definition}\cite{FRANKLARSON}
	Let 	 $\mathcal{E}$ be a Hilbert C*-module over a unital C*-algebra $\mathcal{A}$. A collection $\{\tau_n\}_{n=1}^\infty$ in  $\mathcal{E}$ is said to be a  \textbf{modular Parseval frame} for  $\mathcal{E}$ if 
	\begin{align*}
		 \langle x , x \rangle =\sum_{n=1}^\infty \langle x, \tau_n \rangle \langle  \tau_n, x \rangle,   \quad \forall x \in \mathcal{E}.
	\end{align*}
\end{definition}
As shown in \cite{FRANKLARSON}  a modular Parseval frame $\{\tau_n\}_{n=1}^\infty$ for   $\mathcal{E}$ gives an adjointable isometry 
	\begin{align*}
	\theta_\tau : \mathcal{E} \ni x \mapsto \theta_\tau x \coloneqq   \{\langle x, \tau_n\rangle \}_{n=1}^\infty\in 	\ell^2(\mathbb{N}, \mathcal{A})
\end{align*}
  with adjoint 
\begin{align*}
	\theta_\tau^*: 	\ell^2(\mathbb{N}, \mathcal{A}) \ni \{a_n\}_{n=1}^\infty \mapsto \theta_\tau^*\{a_n\}_{n=1}^\infty\coloneqq \sum_{n=1}^{\infty}a_n \tau_n \in \mathcal{E}.
\end{align*}
We this preliminaries we can derive noncommutative analogue of Theorem \ref{RT}. In the following theorem, given a  subset $\Lambda \subseteq \mathbb{N}$, we set the notation 
\begin{align*}
	o(\Lambda)\coloneqq \text{Number of elements  in }  \Lambda.
\end{align*}
\begin{theorem} (\textbf{Noncommutative Donoho-Stark-Elad-Bruckstein-Ricaud-Torr\'{e}sani Uncertainty Principle})
	For any two   modular Parseval frames  $\{\tau_n\}_{n=1}^\infty$ and $\{\omega_m\}_{m=1}^\infty$ for a Hilbert C*-module $\mathcal{E}$,  we have 
	\begin{align*}
	\left(\frac{\|\theta_\tau x\|_0+\|\theta_\omega x\|_0}{2}\right)^2\geq 	\|\theta_\tau x \|_0 	\|\theta_\omega x \|_0	\geq \frac{1}{\sup_{n, m \in \mathbb{N}} \|\langle \tau_n, \omega_m\rangle \|^2} , \quad \forall x \in  \mathcal{E}, x \neq 0.
	\end{align*}
\end{theorem}
\begin{proof}
Let $x \in \mathcal{E}$ be nonzero. Using Lemma \ref{CAUCHYSCHWARZ} and the  well-known fact in C*-algebra that `norm respects ordering of positive elements', we get 
\begin{align*}
	\|x\|^2&=\|\langle x, x \rangle \|=	\left\| \sum_{n=1}^{\infty} \langle x, \tau_n \rangle\langle  \tau_n, x   \rangle \right\|=\left\| \sum_{n \in \text{supp}(\theta_\tau x)} \langle x, \tau_n \rangle\langle  \tau_n, x   \rangle \right\|\\
	&=\left\| \sum_{n \in \text{supp}(\theta_\tau x)} \left \langle \sum_{m=1}^{\infty} \langle x, \omega_m \rangle \omega_m, \tau_n \right \rangle \left \langle  \tau_n,   \sum_{k=1}^{\infty} \langle x, \omega_k \rangle \omega_k \right \rangle \right\|\\
	&=\left\| \sum_{n \in \text{supp}(\theta_\tau x)} \left \langle \sum_{m\in \text{supp}(\theta_\omega x)} \langle x, \omega_m \rangle \omega_m, \tau_n \right \rangle \left \langle  \tau_n,   \sum_{k\in \text{supp}(\theta_\omega x)} \langle x, \omega_k \rangle \omega_k \right \rangle \right\|\\
	&=\left\| \sum_{n \in \text{supp}(\theta_\tau x)} \left(\sum_{m\in \text{supp}(\theta_\omega x)}\langle  x, \omega_m\rangle \langle \tau_n, \omega_m   \rangle ^*\right)\left(\sum_{k \in \text{supp}(\theta_\omega x)}\langle x,   \omega_k\rangle \langle  \tau_n, \omega_k  \rangle \right)^* \right\|\\
	&\leq \left \|\sum_{n \in \text{supp}(\theta_\tau x)}  \left\|\sum_{m \in \text{supp}(\theta_\omega x)}\langle \tau_n, \omega_m   \rangle\langle \tau_n, \omega_m   \rangle^*\right\|\left(\sum_{k \in \text{supp}(\theta_\omega x)}\langle x,   \omega_k\rangle \langle  \omega_k, x  \rangle \right)\right\|\\
	&\leq \left \|\sum_{n \in \text{supp}(\theta_\tau x)}  \sum_{m \in \text{supp}(\theta_\omega x)}\|\langle \tau_n, \omega_m   \rangle\langle \tau_n, \omega_m   \rangle^*\|\left(\sum_{k \in \text{supp}(\theta_\omega  x)}\langle x,   \omega_k\rangle \langle  \omega_k, x  \rangle \right)\right\|\\
	&\leq \left(\sup_{n,  m \in \mathbb{N}}\|\langle \tau_n, \omega_m\rangle \|^2\right)\left \|\sum_{n \in \text{supp}(\theta_\tau x)}  \sum_{m \in \text{supp}(\theta_\omega x)}1\cdot \left(\sum_{k \in \text{supp}(\theta_\omega  x)}\langle x,   \omega_k\rangle \langle  \omega_k, x  \rangle \right)\right\|\\
	&\leq \left(\sup_{n,  m \in \mathbb{N}}\|\langle \tau_n, \omega_m\rangle \|^2\right)\left \|o(\text{supp}(\theta_\tau x) o(\text{supp}(\theta_\omega x)\left(\sum_{k \in \text{supp}(\theta_\omega  x)}\langle x,   \omega_k\rangle \langle  \omega_k, x  \rangle \right)\right\|\\
	&=\left(\sup_{n,  m \in \mathbb{N}}\|\langle \tau_n, \omega_m\rangle \|^2\right)\|\theta_\tau x\|_0\|\theta_\omega x\|_0\left\|\sum_{k \in \text{supp}(\theta_\omega x)}\langle x,   \omega_k\rangle \langle  \omega_k, x  \rangle \right\|\\
	&=\left(\sup_{n,  m \in \mathbb{N}}\|\langle \tau_n, \omega_m\rangle \|^2\right)\|\theta_\tau x\|_0\|\theta_\omega x\|_0\|	\langle x, x \rangle \|\\
	&=\left(\sup_{n,  m \in \mathbb{N}}\|\langle \tau_n, \omega_m\rangle \|^2\right)\|\theta_\tau x\|_0\|\theta_\omega x\|_0	\|x\|^2.
\end{align*}
By canceling $\|x\|$ we get the stated inequality.
\end{proof}

Using Chebotarev theorem,  in 2005, Tao \cite{TAO, SILBERSTEIN}  improved Theorem \ref{DS} for prime dimensions $d$. 
\begin{theorem}\label{T}
(\textbf{Tao Uncertainty Principle}) \cite{TAO}
For every prime $p$, 
\begin{align*}
 \|h\|_0+\|\widehat{h}\|_0	\geq p+1, \quad \forall h \in \mathbb{C}^p\setminus \{0\}.
\end{align*}
\end{theorem}
In view of Theorem \ref{T} we make the following conjecture. 
\begin{conjecture}
Let $p$ be a prime and $\mathcal{A}$ be a unital   C*-algebra  with invariant basis number property.  Let  $\hat{}: \mathcal{A}^p \to  \mathcal{A}^p$ be the noncommutative Fourier transform defined by 
\begin{align*}
	\widehat{(a_j)_{j=0}^{p-1}}\coloneqq \frac{1}{\sqrt{p}}\left(\sum_{j=0}^{p-1}a_je^{\frac{-2 \pi i j k}{p}}\right)_{k=0}^{p-1}, \quad \forall (a_j)_{j=0}^{p-1} \in \mathcal{A}^p.
\end{align*}
Then 
\begin{align*}\
	\|x\|_0+\|\widehat{x}\|_0	\geq p+1, \quad \forall x \in \mathcal{A}^p\setminus \{0\}.
\end{align*}
\end{conjecture}
 \section{Acknowledgments}
This research was partially supported by the University of Warsaw Thematic Research Programme ``Quantum Symmetries".

  \bibliographystyle{plain}
 \bibliography{reference.bib}

\end{document}